\documentclass[lletter,11pt]{amsart}
\usepackage[colorlinks,citecolor=red,urlcolor=blue,bookmarks=false,hypertexnames=true]{hyperref} 
\usepackage{amsmath, amssymb, amsfonts, amsbsy, amsthm, latexsym, epsfig}
\usepackage{epstopdf }
\usepackage{graphicx}
\usepackage[shortlabels]{enumitem}
\usepackage{color}
\usepackage{lmodern}
\usepackage{microtype}
\usepackage{mathtools}
\mathtoolsset{showonlyrefs}
\definecolor{carlos-color}{rgb}{1.0,0.0,0.0}
\definecolor{applegreen}{rgb}{0.55, 0.71, 0.0}
\definecolor{cadmiumgreen}{rgb}{0.0, 0.42, 0.24}
\definecolor{burntorange}{rgb}{0.8, 0.33, 0.0}

\theoremstyle{plain}
\newtheorem{theorem}{Theorem} [section]
\newtheorem{lemma}[theorem]{Lemma}
\newtheorem{proposition}[theorem]{Proposition}
\newtheorem{corollary}[theorem]{Corollary}

\theoremstyle{definition}

\newtheorem*{teo*}{Theorem}

\theoremstyle{definition}
\newtheorem{definition}[theorem]{Definition}

\newtheorem{remark}[theorem]{Remark}

\newcommand{\esssup}{{\mathrm{ess}\sup}}
\newcommand{\cF}{\mathcal{F}}

\newcommand{\cM}{\mathcal{M}}
\newcommand{\cN}{\mathcal{N}}
\newcommand{\cR}{\mathcal{R}}
\newcommand{\TT}{{\mathbb T}}

\newcommand{\cB}{\mathcal{B}}

\newcommand{\cX}{\mathcal{X}}
\newcommand{\cW}{\mathcal{W}}
\newcommand{\HH}{\mathcal{H}}
\newcommand{\KK}{\mathcal{K}}

\newcommand{\cA}{\mathcal{A}}

\newcommand{\R}{\mathbb{R}}
\newcommand{\Z}{\mathbb{Z}}
\newcommand{\N}{\mathbb{N}}
\newcommand{\CC}{\mathbb{C}}

\newcommand{\la} {\lambda}
\newcommand{\var} {\varepsilon}

\newcommand{\supp}{{\rm supp\,}}

\newcommand{\spn}{{\rm span}}

\begin{document}
	%\nocite{*}
	
	\title{Reducing and Invariant subspaces under two commuting shift operators}
	
	\let\thefootnote\relax

	\author{A. Aguilera}
	\address{ Departamento de Matem\'atica, Universidad de Buenos Aires,
		Instituto de Matem\'atica "Luis Santal\'o" (IMAS-CONICET-UBA), Buenos Aires, Argentina}
	\email{aaguilera@dm.uba.ar}
	
	\author{C. Cabrelli}
	\address{ Departamento de Matem\'atica, Universidad de Buenos Aires,
		Instituto de Matem\'atica "Luis Santal\'o" (IMAS-CONICET-UBA), Buenos Aires, Argentina}
	\email{carlos.cabrelli@gmail.com}
	
	\author{D. Carbajal}
	\address{ Faculty of Mathematics, University of Vienna, Vienna, Austria}
	\email{diana.agustina.carbajal@univie.ac.at}
	
	\author{V. Paternostro}
	\address{ Departamento de Matem\'atica, Universidad de Buenos Aires,
		Instituto de Matem\'atica "Luis Santal\'o" (IMAS-CONICET-UBA), Buenos Aires, Argentina}
	\email{vpater@dm.uba.ar}

	\begin{abstract}
		
		In this article, we characterize reducing and invariant subspaces of  the space of square integrable functions defined in the unit circle and having values in some Hardy space with multiplicity. 
		We consider subspaces that reduce the bilateral shift and at the same time are invariant under the unilateral shift acting locally.
		We also study  subspaces that reduce both operators. The conditions obtained are of the type of  the ones in Helson and Beurling-Lax-Halmos theorems on characterizations of the invariance for the bilateral and unilateral shift.
		The motivations for our study were inspired by recent results on Dynamical Sampling in shift-invariant spaces.

	\end{abstract}

	\maketitle
	\section{Introduction}
	
	Invariant subspaces under shift operators have been studied and characterized  by many authors.
	In particular,  Beurling  in 1949 \cite{Ber49} proved the celebrated theorem in which he characterizes  the invariant subspaces for the unilateral shift acting on the Hardy space $H^2(\TT)$ where $\TT$ is the unit circle  (see Theorem \ref{invariant_subspace_Hardy}). This theorem led to an enormous amount of work and  has stimulated a very fruitful research in different directions.
	
	Many generalizations and applications of this result have enriched the literature on invariant subspaces.
	Furthermore,  the {\it compression of the shift} acting on  the orthogonal complement of any  non-trivial  invariant subspace for the unilateral shift  in $H^2$, has served as a model of a lavish class of Hilbert space operators. This was the initiation of the theory of model spaces 
	\cite{GMR16}.
	Later, Helson and Lowdenslager \cite{HL61} generalized Beurling's result to $L^2(\TT).$

	In 1959,  Lax \cite{Lax59}  extended Beurling's theorem to Hardy spaces of $\KK$-valued functions where $\KK$ is a finite dimensional Hilbert space, i.e. Hardy spaces with finite multiplicity. Shortly after, in 1961
	Halmos \cite{Ha61}  obtained a  characterization
	for the general case (infinite multiplicity), using a beautiful functional analysis approach. See Theorem \ref{thm:beu-lax-hal}, referred to as Beurling-Lax-Halmos Theorem.
	
	On the other hand, generalizations to invariant subspaces of  $L^2(\TT,\KK)$ by the bilateral shift have been given by Helson and Lowdenslager \cite{HL61,He64} and 
	Srinivasan \cite{Sri64}.
	
	In the present paper, we consider two operators  acting on $L^{2}(\TT,H^{2}_{\KK})$: 
	the bilateral shift $U$ and the unilateral shift $\widehat{S}$,  the  latter acting pointwisely on $H^2_{\KK}$. See Definition \ref{def:S-sombrero}.
	We obtain a characterization, in the line of the theorems of Helson, and Beurling-Lax-Halmos,  for the subspaces that are reducing for  $U$  -- that is, invariant under $U$ and $U^*$ -- and at the same time are invariant under  $\widehat{S}.$ We also characterize the subspaces that are reducing for both, $U$ and $\widehat{S}$.
	
	\subsection{Motivation from Dynamical Sampling}
	The motivation for our study comes from the {\em Dynamical Sampling Problem in shift-invariant spaces } formulated in \cite{ACCP21, ACCP22}. For references on the general problem of dynamical sampling see \cite{DS1,DS2,DS3,DS4,DS5}.
	
	Dynamical sampling involves reconstructing an unknown signal that evolves over time from its spatio-temporal samples. A common scenario in dynamical sampling is when the initial spatial samples are insufficient to fully recover the signal, which requires compensating for this lack of information by sampling the signal at the same spatial locations but at different times. In other words, rather than having a large number of sensors that are activated only once, we activate a smaller number of sensors multiple times. This approach is particularly significant  when the cost of sensors is substantial.
	
	The problem of dynamical sampling can be equivalently stated as the problem of determining when the orbit of a function through a bounded operator in a Hilbert space forms a frame.  A surprising result is that this occurs if and only if the operator is similar to the compression of the shift in a model subspace of the Hardy space in the unit disk \cite{DS5}. Lately, there has been a great  interest in this approach, and researchers have also explored the scenario involving multiple orbits. The generalization  to multiple functions requires consideration of model subspaces in Hardy spaces with {\it multiplicity}, as described in \cite{DS6}. Beurling, Halmos, Helson, Lax  and Lowdenslager, among others, have characterized model spaces in these different contexts for both forward and bilateral iterations of the shift operator.
	
	When considering the iterations of two commuting operators, the model space that arises is the orthogonal complement in the space $L^2(\TT, H^2_\KK)$ of a subspace that is reducing  by the bilateral shift and invariant by the unilateral shift acting locally in the Hardy space $H^2_\KK$. Despite their importance, these subspaces have not been previously characterized. 
	The contribution of this paper is the characterization of subspaces that are both invariant and reducing  for the shifts involved. 
	
	The general setting of two commuting operators acting in a Hilbert space $\HH$
	includes the case where  $\HH$ is an integer translation invariant subspace of $L^2(\R),$  one operator is the translation by the integer 1 and the other is a shift-preserving operator that is iterated forwardly. These subspaces play a crucial role in various applications as sampling, wavelet  and approximation theory.

	\subsection{Notation}
	
	In order to state our results in more detail we need to introduce some notation and known facts.
	Along this paper, all the  Hilbert spaces considered will be complex and separable. In particular the letters $\KK$ and $\HH$  will always denote  Hilbert spaces.

	For Hilbert spaces $\HH$ and $\KK$, we will denote by $\cB(\HH,\KK)$ the set of linear bounded operators from $\HH$ into $\KK$ and $\cB(\HH):=\cB(\HH,\HH)$. We will write $\N_0:= \N\cup \{0\}$ and $\TT:=\{z\in \CC\,:\, |z|=1\}$. If $E$ is a  measurable subset of $\TT$, we will denote by $|E|$ its Lebesgue measure normalized such that $|\TT|=1$. 
	We will use the symbol $\oplus$ to denote the orthogonal sum of subspaces. For $\cN,\cM$ closed subspaces of $\HH$, we will use $\cN^{\perp}$ to denote the orthogonal complement of $\cN$ in $\HH$, $\cM\ominus\cN = \cM\cap\cN^\perp$, and $P_{\mathcal N}$ will denote the orthogonal projection of $\HH$ onto $\mathcal N$.  
	Finally, we will write $\cN\simeq_{\Phi} \cM$ if $\cN$ and $\cM$ are isomorphic through a bounded operator $\Phi$.
	
	\subsection{Spaces of vector-valued functions}\label{sec:vector-valued-functions} A vector-valued (or $\KK$-valued) function $f:\TT\rightarrow \KK$ is said to be measurable if for each $x\in\KK$, the complex-valued function $\la\mapsto \langle f(\la),x \rangle_{\KK}$ is measurable on $\TT$. 
	
	The space  $L^{2}(\TT,\KK)$ is the Hilbert space of all measurable $\KK$-valued functions $f$ such that $\int_{\TT}\|f(\la)\|_{\KK}^{2}\,d\la<\infty$,
	endowed with the inner product 
	\begin{equation}
		\langle f,g \rangle = \int_{\TT} \langle f(\la), g(\la) \rangle_{\KK}\,d\la, \qquad f,g\in L^{2}(\TT,\KK).
	\end{equation}
	
	Given  an orthonormal basis $\mathcal{B} = \{\var_{i}\}_{i\in I}$ of $\KK$ (where $\#I=\dim(\KK)$), we can write 
	\begin{equation}
		f(\la) = \sum_{i\in I} \langle f(\la),\var_{i} \rangle_{\KK} \,\var_{i}, \quad \text{a.e. }\la\in\TT.
	\end{equation}
	Let us call $f_{i} := \langle f(\cdot),\var_{i} \rangle_{\KK}$ the {\it coordinate functions} of $f$  respect to the basis $\mathcal{B}$. It can be seen that $f_i$  belongs to $L^{2}(\TT)$ for every $i\in I$. Indeed, $f_{i}:\TT\to \CC$ is measurable since so is $f:\TT\rightarrow \KK$. On the other hand, $|\langle f(\la),\var_{i} \rangle_{\KK}| \leq  \|f(\la) \|_{\KK} \|\var_{i} \|_{\KK} = \|f(\la) \|_{\KK} $ a.e. $\la\in\TT$ and thus we have that
	\begin{equation}
		\|f_{i}\|_{L^{2}(\TT)}^{2} = \int_{\TT} | f_{i}(\la)|^{2}\,d\la= \int_{\TT} |\langle f(\la),\var_{i} \rangle_{\KK}|^{2} \,d\la
		\leq \int_{\TT}  \|f(\la) \|_{\KK}^{2} \,dz = \|f\|
		< \infty.
	\end{equation}
	
	The Hardy space $H^{2}$ will be the subspace of $L^{2}(\TT)$ consisting of all $f\in L^{2}(\TT)$ whose Fourier coefficients vanish for $n<0$, i.e., 
	\begin{equation}
		H^{2} := \left\{ f\in L^{2}(\TT) : \int_{\TT} f(z)z^{-n}\,dz=0 \text{ for } n<0 \right\}.
	\end{equation}
	The Hardy space with multiplicity is denoted by  $H^{2}_{\KK}:=H^{2}(\TT,\KK)$. It is   the closed subspace of  $L^{2}(\TT,\KK)$ consisting of all functions $f\in L^{2}(\TT,\KK)$ whose coordinate functions $f_{i}$ respect to any orthonormal basis of $\KK$ belong to the  Hardy space $H^{2}$. 
	Equivalently, $H^2_\KK$ is the subspace of  functions in $L^{2}(\TT,\KK)$ with zero negative Fourier coefficients, where these latter are defined in a weak sense (for details see \cite[page 48]{RR}).
	
	In this paper we will consider operators acting on the Hilbert space $L^{2}(\TT, H^{2}_{\KK}),$ that is, measurable functions defined in the circle $\TT,$
	and having values in the space  $H^{2}_{\KK}$. Since there are two variables involved we will stablish the following convention:
	we will use the letter $\lambda$ for the variable of  functions in $L^2(\TT,\HH)$ for any Hilbert space $\HH$ and $z$ for the variable  of functions in $H^2_\KK$. 
	So, 
	for $f\in L^2(\TT,H^2_\KK), \, f(\lambda) \in H^2_\KK \text{ and } f(\lambda)(z) \in \KK$ for a.e. $\lambda,z \in \TT$.

	\subsection{Shift-operators} First, let us define two operators which will play a crucial role in this paper: the bilateral and the unilateral shifts. 
	
	\begin{definition}\label{def:U}
		The operator $U:L^{2}(\TT,\KK)\to L^{2}(\TT,\KK)$ defined by 
		\begin{equation}\label{eq:def_U}
			(U f)(\la) = \la f(\la),\quad  \text{a.e. }\la\in\TT,\,f\in L^{2}(\TT,\KK),
		\end{equation}
		is called the {\it bilateral shift on $L^2(\TT,\KK)$} with multiplicity $\alpha=\dim(\KK)$.
	\end{definition}
	
	Observe that $U$ is unitary and its adjoint operator is given by $(U^{*} f )(\la) = \overline{\la} f(\la)$, for a.e. $\la\in\TT$ and $f\in L^{2}(\TT,\KK)$.			
	
	\begin{definition}\label{def:S}
		The operator $S: H^{2}_{\KK} \rightarrow H^{2}_{\KK}$ given by the restriction of $U$ to $H^{2}_{\KK}$  is called the \textit{unilateral shift on  $H^{2}_{\KK}$} with multiplicity $\alpha=\dim(\KK)$.
	\end{definition}
	Since $H^{2}_{\KK}$ is invariant under $U$ which  is unitary, then the operator $S$ is an isometry. 
	When $\KK=\CC$, we have that $L^2(\TT, \KK)=L^2(\TT)$ and $H^{2}_{\KK}=H^2$, hence the definitions above apply to those spaces accordingly, being the multiplicity $\alpha=1$.

	In the next definition, we will introduce an operator   on $L^{2}(\TT,H^{2}_\KK)$  which  acts \textit{pointwisely} as a unilateral shift operator. 	
	\begin{definition}\label{def:S-sombrero}
		The operator $\widehat{S}: L^{2}(\TT, H^{2}_{\KK}) \rightarrow L^{2}(\TT, H^{2}_{\KK})$ is defined by
		\begin{equation}
			(\widehat{S}f)(\la)=S(f(\la)),\quad \text{a.e. } \la\in \TT,\, f\in L^2(\TT,H^{2}_{\KK}),
		\end{equation}
		where $S$ is the unilateral shift operator on $H^2_\KK$. More precisely, 
		for 	$f\in L^2(\TT,H^{2}_{\KK})$ and for a.e. $\lambda,z\in\TT$
		$$
		(\widehat{S}f)(\la)(z)=S(f(\la))(z)=zf(\lambda)(z).
		$$
	\end{definition}
	The choice of notation $\widehat{S}$ will be clear in Subsection~\ref{sec:inv-subspaces} (cf. Remark \ref{rem:S-sombrero}). 
	An easy computation shows that $\widehat{S}$ and $U$, the bilateral shift on $L^2(\TT,H^2_\KK)$, commute. Moreover, $\widehat{S}$ is an isometry.
	A subspace of $L^2(\TT,H^2_\KK)$ that is reducing for $U$ and $\widehat{S}$ will be called a {\it full-Hardy} space,
	see Subsection \ref{full-Hardy}.

	There is a natural way to construct an orthonormal basis  of $L^2(\TT,\KK)$ or $H^2_\KK$ starting from an orthonormal basis $\cB=\{\var_i\}_{i\in I}$  of $\KK$. Since $\KK$ can be thought as a subspace of $L^2(\TT,\KK)$ (the constant $\KK$-valued functions), we will make an abuse of notation by considering $\{{\var}_i\}_{i\in I}\subset L^2(\TT,\KK)$ as the constant functions with values forming an orthonormal basis of $\KK$. Then, it is easy to see that $\{ U^k{\var}_i\,:\,k\in\Z,\, i\in I\}$ is an orthonormal basis of $L^2(\TT,\KK)$. Analogously, the system $\{S^j{\var}_i\,:\,j\in\N_0,\, i\in I\}$ is an orthonormal basis of~$H^2_\KK$.
	\subsection{Main results}
	We are now ready to state our main results.
	The first one is the characterization of subspaces of $L^2(\TT,H^2_\KK)$ reducing for $U$ and invariant under $\widehat{S}$.
	\begin{theorem}\label{characterization_hatS_inv}
		Let $\mathcal{M}\subseteq L^{2}(\TT,H^{2}_{\KK})$ be a closed subspace. The following statements are equivalent:
		\begin{enumerate}
			\item[\rm i)] $\mathcal{M}$ is reducing for $U$ and invariant under $\widehat{S}$.
			\item[\rm ii)] There exists  a full-Hardy subspace $\cW\subseteq L^{2}(\TT,H^{2}_{\KK})$ and a partial isometry $\Phi:L^{2}(\TT,H^{2}_{\KK}) \rightarrow L^{2}(\TT,H^{2}_{\KK})$ with initial space $\cW$  that commutes with $U$ and $\widehat{S}$
			such that $\cW \simeq_{\Phi} \cM.$
		\end{enumerate}
	\end{theorem}
	
	Our second result gives the uniqueness of the characterization up to an isometry that commutes with $U$ and $\widehat{S}.$
	
	\begin{theorem}\label{thm:uniqueness}
		Let $\cW_1,\cW_2\subseteq L^2(\TT,H^2_\KK)$ be two full-Hardy spaces and $\Phi_1,\Phi_2:L^2(\TT,H^2_\KK)\to L^2(\TT,H^2_\KK)$ two partial isometries with initial spaces  $\cW_1$ and $\cW_2$, respectively. Suppose that $\Phi_1$ and  $\Phi_2$  commute with $U$ and $\widehat{S}$, and that $\Phi_1(\cW_1)= \Phi_2(\cW_2)$.  Then, there exists a partial isometry $\Psi :L^{2}(\TT,H^{2}_{\KK}) \rightarrow L^{2}(\TT,H^{2}_{\KK})$ with initial space $\cW_2$ that commutes with $U$ and $\widehat{S}$ such that $\Psi(\cW_2)=\cW_1$ and $\Phi_2=\Phi_1\Psi$.
	\end{theorem}
	
	To state our theorem on the  characterization of  subspaces of $L^2(\TT,H^2_\KK)$ that simultaneously reduce $U$ and 
	$\widehat{S}$ we need more notation, this is given in Section \ref{full-Hardy}. 
	
	We want to remark here that our results hold in the more general setting of
	{\it multiplication-invariant spaces} \cite{Sri64, BR}, where the circle $\TT$ with the normalized Lebesgue measure is replaced by a general finite measure space $(X,\mu)$. In particular, this includes the multivariable case of $L^2(\TT^d,H^{2}_{\KK}).$ See Subsection  \ref{MIS}. We stick here with the circle $\TT$ for
	simplicity in the treatment.
	
	About Theorem \ref{characterization_hatS_inv}: The proof of this theorem uses fiberization techniques, a standard tool for reducing subspaces  of $L^2(\TT,\KK)$. 
	Informally, a subspace $\cM$ of $L^2(\TT,H^2_\KK)$ that reduces $U$
	can be described, using a result of Helson, by a {\it measurable range function}, that is, a family of subspaces
	of $H^2_\KK$ indexed by $\TT$, say $\{J_{\cM}(\lambda) \subseteq H^2_\KK: \la \in \TT\}$, `measurable' in a sense that will be defined later. On the other hand, the invariance of $\cM$ under $\widehat{S}$ implies that each {\it fiber space }$J_\cM(\lambda)$ is invariant under the unilateral shift in $H^2_\KK$ for almost every $\lambda$.
	
	The fiberization technique consists in solving the problem in almost each fiber space $J_\cM(\la)$ and transferring the results 
	back to the main space, in our case, from $\{J_\cM(\lambda)\}_\la$ to $\cM$.
	Thus, one can think that by applying the Beurling-Lax-Halmos Theorem to $J_\cM(\lambda)$ for almost each $\lambda \in \TT$ we can deduce the desired characterization. However, this strategy does not work since measurability issues come up and it does not seem possible
	to solve it within this context. Instead, our approach is to construct carefully the needed isometries,  step by step, taking into account the measurability property.
	
	The organization of the paper is as follows. In Section \ref{sec:inv-subspaces} we state classical results on the characterization of invariant and reducing  subspaces of $L^2(\TT,\KK)$ by the bilateral shift $U$ and the invariant subspaces of $H^{2}_{\KK}$ by the unilateral shift $S$. In particular, we review the theory of {\it range fuctions} and list some of their properties, since it is one of the main tools for our characterization. Section \ref{our-theorems} is devoted to the proof of our main results.
	In particular, in Subsection \ref{full-Hardy} we characterize the spaces reducing for $U$ and $\widehat{S}$, the full-Hardy spaces. The proof of the characterization theorem for subspaces reducing for $U$ and invariant under $\widehat{S}$ is in Subsection \ref{characterization} and in Subsection \ref{uniqueness} we prove the uniqueness of the characterization.
	
	\section{Invariant subspaces}\label{sec:inv-subspaces}
	
	In this section we will state some well known results on invariant subspaces. For further details, we refer the reader to \cite{RR, Co}.
	
	Let $\HH$ be a Hilbert space, $A \in \cB(\HH)$ and $\cM\subseteq \HH$ a closed subspace. The subspace $\cM$ is {\it invariant} for $A$ (or $A$-invariant) if $A(\cM)\subseteq \cM$. Furthermore, $\cM$ is {\it reducing} for $A$ (or $\cM$ {\it reduces }$A$) if $\cM$ and $\cM^\perp$ are invariant  under $A$.	 The condition $A(\cM^{\perp})\subseteq \cM^{\perp}$ is equivalent to $A^{*}(\cM)\subseteq \cM$, where $A^{*}$ denotes the adjoint operator of $A$. Hence, $\cM$ is reducing for $A$ if and only if $\cM$ is invariant under $A$ and $A^*$.

	\subsection{Wandering subspaces}		
	Given a Hilbert space $\HH$ and an operator $A\in \cB(\HH)$, a \textit{wandering subspace} $\cR\subseteq \HH$ for $A$ is a subspace such that $\cR\perp A^j(\cR)$ for every $j\geq 1$. When $A$ is an isometry, this latter condition is equivalent to $A^j(\cR)\perp A^{j'}(\cR)$ for every $j,j'\geq 0$ and $j\neq j'$.
	
	The wandering subspaces are related to the invariant subspaces for $A$ as follows: 
	Each wandering subspace $\cR$ produces an invariant subspace for $A$ of the form $\cM_\cR = \bigoplus_{j=0}^{\infty} A^j(\cR)\subseteq\HH$.
	In fact, the correspondence $\cR \mapsto \cM_{\cR}$ between wandering subspaces and invariant subspaces for $A$ is one-to-one, since $\cR=\cM_{\cR}\ominus A(\cM_{\cR})$. 
	
	On the other hand, given an invariant subspace $\cM\subseteq\HH$ for $A$, the subspace $\mathcal R=\cM \ominus A(\cM)$ is a wandering subspace for $A$ and the equation below holds
	\begin{equation}\label{eq:decomp-M-inv-for-A}
		\cM = \bigoplus_{j=0}^{\infty} A^j(\cR) \oplus \bigcap_{j=0}^{\infty} A^j (\cM).
	\end{equation}
	
	We say that $A$ is a \textit{pure isometry} if $\bigcap_{j=0}^{\infty} A^j(\HH)=\{0\}$. For this type of operators, the decomposition in \eqref{eq:decomp-M-inv-for-A} reduces to the following.
	
	\begin{lemma}\label{lem:A-pure}
		Let $\HH$ be a Hilbert space and $A\in\cB(\HH)$ a pure isometry. Then  if $\cM\subseteq \HH$ is an invariant subspace for $A$, we have that 
		\begin{equation}\label{eq:decomp-M-A-pure}
			\cM=\bigoplus_{j=0}^{\infty} A^j(\cR),
		\end{equation}
		where $\cR = \cM \ominus{A}(\cM)$ is the unique wandering subspace for $A$ that satisfies~\eqref{eq:decomp-M-A-pure}.
	\end{lemma}
	
	\begin{proof}
		Observe that $\bigcap_{j=0}^{\infty} A^j (\cM) \subseteq \bigcap_{j=0}^{\infty} A^j (\HH)=\{0\}$, and so from \eqref{eq:decomp-M-inv-for-A} we deduce~\eqref{eq:decomp-M-A-pure}.
	\end{proof}
	
	\begin{lemma}\label{lem:isometry-inv-wan}
		For $i=1,2$, let $\HH_i$ be a Hilbert space and let $\cM_i\subseteq \HH_i$ be an invariant subspace for the pure isometry $A_i\in\cB(\HH_i)$, with associated wandering subspace $\cR_i$. Then, if $\Phi:\HH_1\to\HH_2$ is an isometry such that $\Phi A_1 = A_2 \Phi$, we have that
		$\cM_1\simeq_{\Phi} \cM_2$ if and only if $\cR_1\simeq_{\Phi} \cR_2$.
	\end{lemma}
	
	\begin{proof}
		Observe that 
		\begin{equation}\label{eq:Phi(M1)}
			\Phi(\cM_1)= \Phi(\bigoplus_{j=0}^\infty A_1^j(\cR_1)) = \bigoplus_{j=0}^\infty A_2^j\Phi(\cR_1).
		\end{equation} If $\Phi(\cM_1)=\cM_2$, by uniqueness of the wandering subspace for $A_2$ associated to $\cM_2$, we get that $\Phi(\cR_1)=\cR_2$. Conversely, if $\Phi(\cR_1)=\cR_2$, from \eqref{eq:Phi(M1)} we see that $\Phi(\cM_1)=\cM_2$.
	\end{proof}
	
	It can be easily seen that the unilateral shift $S$ is a pure isometry and moreover, the following lemma holds. The proof can be found in \cite[Lemma 4]{Ha61} (see also \cite[Lemma 3.24]{RR}).
	
	\begin{lemma}\label{lem:dimR_0}
		Let $\KK$ be a Hilbert space and let $\cM\subseteq H^{2}_{\KK}$ be a closed subspace which is invariant under $S$. Then, the wandering subspace $\mathcal R=\cM \ominus S\cM \subseteq H^{2}_{\KK}$ satisfies that
		$\dim(\cR) \leq \dim(\KK)$.
	\end{lemma}
	
	\subsection{Reducing subspaces  for the bilateral shift of $L^2(\TT,\KK)$}\label{sec:reducing-subspaces}
	
	In this subsection, we will focus on the reducing subspaces for the bilateral shift $U$ of Definition \ref{def:U} acting on $L^2(\TT,\KK)$. These subspaces were completely characterized by Helson \cite{He64} in terms of \textit{range functions}. 
	
	A range function $J$ in $\KK$ is a mapping
	$$J:\TT \rightarrow \{ \text{closed subspaces of } \KK \}.$$
	The range function $J$ is measurable if $\la \mapsto \langle P_{J(\la)}x,y \rangle $ is a measurable complex-valued function for each $x,y \in\KK$, where $P_{J(\la)}$ denotes the orthogonal projection of $\KK$ onto $J(\la)$.

	The characterization theorem proved in \cite{He64} was later extended by Bownik and Ross in \cite[Theorem 2.4]{BR}, as we will state it. We remark that the subspaces of $L^2(\TT,\KK)$ which are reducing for $U$ are \textit{multiplication invariant} subspaces of $L^2(\TT,\KK)$ as defined in~\cite[Definition 2.3]{BR}, see also Definition \ref{def:mult-inv}. 
	
	\begin{theorem}\label{reducing_U}
		A closed subspace $\mathcal{M}\subseteq L^{2}(\TT,\KK)$ is reducing for $U$ if and only if there exists a measurable range function $J$ such that
		\begin{equation}\label{eq:charac-M-reducing}
			\mathcal{M}=\{f\in L^{2}(\TT,\KK) : f(\lambda)\in J(\lambda) \text{ a.e }\lambda\in \TT\}.
		\end{equation}
		Identifying range functions which are equal almost everywhere, the correspondence between reducing subspaces for $U$ and measurable range functions is one-to-one and onto.
		
		Moreover, if $\mathcal{M}$ is generated by iterations of $U$ on an (at most) countable set of functions ${\cA\subseteq L^{2}(\TT,\KK)}$, i.e. $\cM = \overline{\spn}\{U^{k}f :f\in\cA,\,k\in\Z\},$ then, the measurable range function associated to $\cM$ is given by
		$J(\lambda)= \overline{\spn}\{f(\lambda):f\in \cA\}$, for a.e. $\lambda\in\TT$.
	\end{theorem} 
	
	Since $L^2(\TT,\KK)$ is a separable Hilbert space, every closed subspace $\mathcal{M}\subseteq L^2(\TT,\KK)$ reducing for $U$ admits an at most countable \textit{set of generators} $\cA\subset L^2(\TT,\KK)$ in the above sense.
	
	Observe that for the particular case of $\KK=\CC$, since its closed subspaces  are $\CC$ or $\{0\}$, a range function $J$ in $\CC$ must be of the form $J(\la)=\cX_{E}(\la) \CC$ for a.e. $\la\in\TT$, where $E$ is a measurable set in $\TT$ and $\chi_E$ denotes its characteristic function.  Hence, Theorem \ref{reducing_U} implies that the reducing subspaces for $U$ of $L^2(\TT)$ are the ones that can be written as 
	$$\cM=\{ f\in L^{2}(\TT) : f=0 \text{ a.e. } \TT\setminus E \} = \cX_{E} L^{2}(\TT).$$ 
	
	Given that the reducing subspaces for $U$ of $L^2(\TT,\KK)$ can be characterized through the \textit{associated} range function, many properties of these spaces can be understood through the pointwise properties of the range function (as long as they are satisfied uniformly) as we see below. From now on, we will use the notation $J_{\cM}$ when the range function is associated to a reducing subspace $\cM$. 
	
	\begin{lemma}\label{lem:props-J}
		The following statements hold.
		\begin{enumerate}
			\item[\rm i)] Let $\cM$ be a reducing subspace for $U$ in $L^2(\TT,\KK)$ with range function $J_\cM$. Then,  the  range function of $\cM^\perp$ is given by $\la \mapsto (J_\cM(\la))^\perp$, for a.e. $\la\in\TT.$
			\item[\rm ii)] Let  $\{\cM_n\}_{n\in I}$ be an at most countable  sequence of pairwise orthogonal  reducing subspaces for $U$ in $L^2(\TT,\KK)$  with range functions $\{J_{\cM_n}\}_{n\in I}$. Then, $\bigoplus_{n\in I} \cM_n$ is reducing for $U$ and its range function is given by $\la\mapsto \bigoplus_{n\in I} J_{\cM_n}(\la)$, for a.e. $\la\in\TT$. 
		\end{enumerate}
	\end{lemma}
	
	The proof of the above lemma is a  consequence of the results in \cite{BR}.

	Since the range function $J_{\cM}(\la)$ can be the subspace $\{0\}$ for some values of $\la\in\TT$, we define the {\it spectrum of $\cM$} as the measurable set
	\begin{equation}\label{eq:specturm}
		\sigma(\cM) = \left\{\la\in\TT : J_{\cM}(\la) \neq \{0\} \right\}.
	\end{equation}
	Even more, the dimension of  $J_\cM(\la)$ as a subspace of $\KK$ may vary with $\lambda\in\TT$. Thus, it is often convenient to analyze the subsets of $\TT$ where this dimension is constant. The following lemma shows that given a closed subspace $\cM \subseteq L^{2}(\TT,\KK)$ reducing for $U$ with measurable range function $J_{\cM}$, there exists a measurable partition of $\TT$ into sets where $\dim (J_{\cM}(\lambda))$ is constant and explicitly provides a \textit{measurable} orthonormal basis of $J_{\cM}(\lambda)$ which holds a.e. over each of these sets. This result is a consequence of \cite[Theorem 2.6]{BR} and the proof can be easily extended from \cite[Lemma 2.9]{ACCP1}.  
	
	\begin{lemma}\label{measurable_sets}
		Let $\cM \subseteq L^{2}(\TT,\KK)$ be a reducing subspace for $U$ with range function $J_{\cM}$. For each $n\in \N_0\cup\{\infty\}$ define the sets $A_n=\{\la\in\TT\,:\,\dim (J_{\cM}(\lambda))=n \}$. Then,  $\{A_n\}_n$ is a family of disjoint measurable sets such that $\TT = \bigcup_{n\in\N_{0}\cup\{\infty\}} A_{n}$ and there exist functions $\{\phi_{i}\}_{i\in\N} \subset L^{\infty}(\TT,\KK)$ for which the following statements hold:
		\begin{enumerate}
			\item[\rm i)] $\{U^{k}\phi_{i} : i\in\N, k\in\Z \}$ is a Parseval frame of  $\cM$.
			\item[\rm ii)] For every $n\in\N$ and $i>n$, $\phi_{i}(\la) = 0$ a.e. $\la \in A_{n}$.
			\item[\rm iii)] For every $n\in\N$, $\{ \phi_{1}(\la),...,\phi_{n}(\la)\}$ is an orthonormal basis of $J_{\cM}(\la)$ a.e. $\la\in A_{n}$, and $\{ \phi_{i}(\la)\}_{i\in\N}$ is an orthonormal basis of $J_{\cM}(\la)$ a.e. $\la\in A_{\infty}$.
		\end{enumerate}
	\end{lemma}
	
	\begin{remark}\label{rem:finite-An}
		Observe that if $\dim(J_\cM(\la))\leq k$ for a.e. $\la\in\TT$, then $|A_n|=0$ for every $n>k$ and thus $\phi_i\equiv 0$ for every $i>k$. Therefore, in this case, we can assume that the partition $\{A_n\}_n$ of $\TT$ and the functions $\{\phi_i\}_i$ are finitely many (as many as $k$ or less).
	\end{remark}
	
	\subsection{Invariant subspaces for the unilateral shift of $H^2_\KK$} Now, we turn to the invariant subspaces for the unilateral shift $S$ of Definition \ref{def:S} acting on $H^2_\KK$. We begin by considering the case of multiplicity $1$, i.e. taking $\KK=\CC$. In this case, the invariant subspaces for $S$ were characterized by Beurling \cite{Ber49}. A simplified version of the proof can be found in \cite[Corollary 3.11]{RR}. 
	
	\begin{definition}
		A function $h\in H^{2}$ is said to be {\it inner} if $|h(z)|=1$ a.e. $z\in\TT$.
	\end{definition}
	
	\begin{theorem}{\rm (Beurling's Theorem).}\label{invariant_subspace_Hardy}\label{thm:beurling}
		A closed subspace $\cM\subseteq H^{2}$ is invariant under $S$ if and only if there exists an inner function $h\in H^{2}$ such that $\cM=hH^{2}$.
		Furthermore, it holds that $h_{1}H^{2}=h_{2}H^{2}$ with $h_{1},h_{2}$ inner functions if and only if $h_{1}/h_{2}$ is a constant a.e. $\la\in\TT$.
	\end{theorem}
	
	An extension of Beurling's theorem (Theorem \ref{thm:beurling}) to multiplicities greater than $1$ was provided by Lax \cite{Lax59}, for the case of finite multiplicity, and by Halmos \cite{Ha61} in the general case. We will refer to this extension as Beurling-Lax-Halmos Theorem. The latter shows that the closed subspaces of $H^2_\KK$ which are invariant under $S$ can be characterized through the theory of \textit{operator-valued functions}. 
	
	Before stating it, let us give some definitions and preliminary results. An operator-valued function in $\KK$ is a function $F:\TT\rightarrow \cB(\KK)$.
	The function $F$ is said to be measurable if for every $x\in\KK$, the $\KK$-valued function $\lambda\mapsto F(\lambda)x$ is measurable. Note that the measurablity of a range function $J$ in $\KK$ is equivalent to the measurability of the operator-valued function in $\KK$ given by $\la\mapsto P_{J(\la)}.$
	
	For an  operator-valued function in $\KK$, $F:\TT\rightarrow \cB(\KK)$,  its norm is defined as  $\|F\|_{\infty} = \esssup_{\lambda\in\TT} \|F(\lambda)\|_{op}.$ 
	
	\begin{definition}
		Let $F:\TT\to\cB(\KK)$ be a measurable operator-valued function in $\KK$ such that $\|F\|_{\infty}<\infty$. We denote by $\widehat{F}:L^{2}(\TT,\KK)\rightarrow L^{2}(\TT,\KK)$ the operator given by
		\begin{equation}
			(\widehat{F}f)(\lambda) = F(\lambda)f(\lambda),\quad \text{a.e. }\lambda\in\TT,\, f\in L^{2}(\TT,\KK),
		\end{equation}
		As $\|F\|_{\infty}<\infty$, we have that $\widehat{F}$ is well-defined and bounded.
	\end{definition}

	\begin{remark}\label{rem:S-sombrero}
		The operator $\widehat{S}:L^2(\TT,H^2_\KK)\to L^2(\TT,H^2_\KK)$ in Definition \ref{def:S-sombrero} is, in fact, the operator associated to the constant operator-valued function $\la\mapsto S$, where $S:H^2_\KK\to H^2_\KK$ is the unilateral shift operator on $H^2_\KK$.
	\end{remark}
	
	Let us denote by $\cF$ the class of all measurable functions $F:\TT\rightarrow \cB(\KK)$ such that $\|F\|_{\infty}<\infty$ and $\widehat{\cF}=\{ \widehat{F} : F \in \cF \}$.
	
	In \cite[Theorem 3.17]{RR} it is shown that the mapping $\cF\rightarrow \widehat{\cF}$ defined by $F\mapsto\widehat{F}$ is an adjoint-preserving algebra isomorphism. In particular, $\widehat{F}$ is normal (self-adjoint, unitary or a projection), if and only if $F(\lambda)$ is normal (self-adjoint, unitary or a projection) for a.e. $\lambda\in \TT$.
	
	An operator-valued function $F \in \cF$ is \textit{analytic} if $\widehat{F}(H^{2}_{\KK}) \subseteq H^{2}_{\KK}$.  We will denote by $\mathcal{F}_{0}$ the set of all analytic elements of $\mathcal{F}$.
	
	Beurling-Lax-Halmos Theorem shows that any closed subspace of $H^2_\KK$ that is invariant under $S$ is isometrically isomorphic to a space of the form $H^2_{\KK_1}$ where $\KK_1\subseteq \KK$ is a closed subspace, through an operator that is associated to an analytic operator-valued function. We refer the reader to \cite[Corollary 3.26]{RR} for a proof.
	
	\begin{theorem}{\rm (Beurling-Lax-Halmos Theorem)}\label{thm:beu-lax-hal}
		A closed subspace $\cM\subseteq H^{2}_{\KK}$ is invariant under $S$ if and only if there exists a subspace $\KK_{1} \subseteq \KK$ and an operator-valued function $F \in \mathcal{F}_{0}$  such that $$\cM=\widehat{F}(H^{2}_{\KK_{1}}),$$ where  $F(z)$ is a partial isometry with initial space $\KK_{1}$ for a.e. $z\in \TT$.
	\end{theorem}

	In the rest of this section, we state some results regarding operator-valued functions that we will need later.
	The following proposition shows that every operator that commutes with $U$ or $S$ must belong to the class $\widehat{\cF}$.
	
	\begin{proposition}\cite[Corollary 3.19 and 3.20]{RR}\label{commutant_U_S}\
		Given a Hilbert space $\KK$, let $U$ and $S$ be the bilateral and unilateral shifts acting on $ L^{2}(\TT,\KK)$  and $H^{2}_{\KK}$ respectively. Then, the following conditions hold:
		\begin{enumerate} 
			\item[\rm i)] The commutant of $U$ is $\widehat{\cF}$.
			\item[\rm ii)] The commutant of $S$ is $\{ \widehat{F}|_{H^{2}_{\KK}} : F\in \cF_{0}\}$.
		\end{enumerate}
	\end{proposition}

	\begin{lemma}\label{lem:props-range-op}
		Let $F:\TT\to \cB(\KK)$ be an operator-valued function in $\cF$, let $A=\widehat{F}$, and let $\cM\subseteq L^2(\TT,\KK)$ be a reducing subspace for $U$ with range function $J_\cM$. Then the following statements hold:
		\begin{enumerate}
			\item[\rm i)] The subspace $\overline{A(\cM)}\subseteq L^2(\TT,\KK)$ is reducing for $U$ and its range function is given by $\la\mapsto \overline{F(\la)(J_\cM(\la))},$ for a.e. $\la\in\TT$.
			\item[\rm ii)] The subspace $\ker(A)\subseteq L^2(\TT,\KK)$ is reducing for $U$ and its range function is given by $\la\mapsto\ker(F(\la))$, for a.e. $\la\in\TT$.
			\item[\rm iii)] $A$ is a partial isometry with initial space $\cM$ if and only if $F(\la)$ is a partial isometry with initial space $J_\cM(\la)$ for a.e. $\la\in\TT$.
			\item[\rm iv)] The subspace $\cR=\cM \ominus \overline{A(\cM)}$ is also reducing for $U$ and its range function is given by $\la\mapsto J_\cM(\la)\ominus F(J_\cM(\la))$, for a.e. $\la\in\TT$.
		\end{enumerate}
	\end{lemma}
	
	The proofs of items i)-iii) of the lemma above can be found in \cite[Theorem 4.1 and Lemma 4.2]{BI} in the context of \textit{multiplication invariant operators} and \textit{range operators} as defined in  Definition 3.2 and Definition 3.6 in \cite{BI}, see also Subsection \ref{MIS}. Item iv) can be easily deduced from i) and Lemma \ref{lem:props-J}.

	\section{Invariant subspaces of $L^2(\TT, H^2_\KK)$}\label{our-theorems}
	
	In this section we will prove our main result, i.e.
	the characterization of the closed subspaces of $L^2(\TT,H^2_\KK)$ which are reducing for the bilateral shift $U:L^2(\TT,H^2_\KK)\to L^2(\TT,H^2_\KK)$ and invariant under $\widehat{S}:L^2(\TT,H^2_\KK)\to L^2(\TT,H^2_\KK)$ (see Definition \ref{def:U} and Definition \ref{def:S-sombrero}, respectively). 
	
	By Theorem \ref{reducing_U}, we know that given a closed subspace $\cM\subseteq L^2(\TT,H^2_\KK)$ reducing for $U$, there exists a measurable range function $J_\cM$ in $H^2_\KK$ associated to $\cM$. The invariance of $\cM$ by $\widehat S$  is equivalent to the invariance for the unilateral shift $S$ on $H^2_\KK$ of $J_\cM(\la)$ for a.e. $\la\in\TT$ as we show next.

	\begin{lemma}\label{lem:invariance_J_S}
		Let $\cM \subseteq L^{2}(\TT,H^{2}_{\KK})$ be a reducing subspace for $U$ with range function $J_\cM$ in $H^2_\KK$. Then, $\cM$ is invariant under $\widehat{S}$ if and only if $J_\cM(\la)$ is invariant under $S$ a.e. $\la \in \TT$.
		
	\end{lemma}
	
	\begin{proof}
		Let $\cA\subset L^2(\TT,H^2_\KK)$ be an at most countable set of generators for $\cM$, that is, $\cM = \overline{\spn}\{U^{k}f :f\in\cA,k\in\Z\}.$ Theorem \ref{reducing_U} states that for a.e. $\la\in\TT$ $$J_\cM(\la)=\overline{\spn}\{f(\la)\,:\,f\in\cA\}.$$
		If $\cM$ is invariant under $\widehat{S}$, then for every $f\in\cA$ we have that $\widehat{S}f\in\cM$. Then $(\widehat{S}f)(\la) = S(f(\la))\in J_\cM(\la)$ for a.e. $\la\in\TT$. Finally, by linearity and continuity of $S$,
		$$S(J_\cM(\lambda)) 
		\subseteq \overline{\spn}\{S(f(\lambda)):f\in\cA\}\subseteq J_\cM(\lambda)$$
		for a.e. $\la\in\TT$.
		The converse is clear.
	\end{proof}
	
	Having in mind Lemma \ref{lem:invariance_J_S}, it appears that we may derive a description of the closed subspaces $\cM\subseteq L^2(\TT,H^2_\KK)$ which are reducing for $U$ and invariant under $\widehat{S}$ by an application of the Beurling-Lax-Halmos Theorem (Theorem \ref{thm:beu-lax-hal}) to each subspace $J_\cM(\la)\subseteq H^2_\KK$ for a.e. $\la\in\TT$. However, this approach is not as simple since, for a.e. $\la\in\TT$, one would have a different operator-valued function $F_\la:\TT\to\cB(\KK)$ and a different initial space $\KK_\la\subseteq\KK$. A characterization for $\cM$, requires the {\it measurability } of the functions $\la \rightarrow F_\la $ and $\la\rightarrow \KK_{\la}$ plus certain \textit{uniformity}. In order to fulfill this requirements our strategy is to construct these functions step by step using the properties of $U$ and $\widehat{S}$.

	\begin{remark}\label{rem:R_0-S-sombrero}
		\
		
		i) Observe that since $S$ is a pure isometry, so is $\widehat{S}$. Now, let $\cM \subseteq L^{2}(\TT,H^{2}_{\KK})$ be a subspace reducing for $U$ and invariant under $\widehat{S}.$ Then, $\cR = \cM \ominus \widehat{S}\cM$ is a wandering subspace for $\widehat{S}$ and by Lemma \ref{lem:A-pure} we have that 
		$$\cM = \bigoplus_{j=0}^{\infty} \widehat{S}^j(\cR).$$ 
		Moreover, item iv) of Lemma \ref{lem:props-range-op} shows that $\cR$ is reducing for $U$. In fact, its range function $J_{\cR}$ produces the wandering subspaces for $S$ associated to $J_\cM(\la)$ for a.e. $\la\in\TT$. Hence, by ii) in Lemma  \ref{lem:props-J} 
		\begin{equation}\label{eq:oplus-S^jR_0}
			J_{\cM}(\la) = \bigoplus_{j=0}^{\infty} S^j(J_{\cR}(\la))
		\end{equation}  and $\dim(J_{\cR}(\la))\leq \dim(\KK)$  
		for a.e. $\la\in\TT$ (see Lemma \ref{lem:dimR_0}).
		
		\vspace{0.5em}
		
		ii) For $i=1,2$, let $\cM_i\subseteq L^2(\TT,H^2_\KK)$ be an invariant subspace for $\widehat{S}$ and $\cR_i$ the wandering subspace for $\widehat{S}$ associated to $\cM_i$. Assume that there exists an isometry $\Phi:L^2(\TT,H^2_\KK) \to L^2(\TT,H^2_\KK)$ such that $\Phi$ commutes with $\widehat{S}$ and $U$ and $\cM_1\simeq_\Phi\cM_2$. Then, by Lemma \ref{lem:isometry-inv-wan}, we have that $\cR_1\simeq_\Phi\cR_2$. Now, as $\Phi$ commutes with $U$, by Proposition \ref{commutant_U_S}, $\Phi\in \widehat{\cF}$, that is, $\Phi = \widehat{F}$ for $F:\TT\to \cB(H^2_\KK)$ an operator-valued function in the class $\cF$. Moreover, using Lemma \ref{lem:props-range-op} we get that $F(\la)$ is an isometry such that $J_{\cR_1}(\la)\simeq_{F(\la)} J_{\cR_2}(\la)$ for a.e. $\la\in\TT$.
	\end{remark}
	
	%%%%%%%%%%%%%%%%%%%%%%%%%%%%%%%%%%	
	\subsection{Full-Hardy spaces}\label{full-Hardy}
	In order to understand the structure of spaces that are reducing for $U$ and invariant under $\widehat S$, we first study a subclass of them, the {\it full-Hardy spaces}. As we will see, these are precisely the spaces that are reducing for both, $U$ and $\widehat S$. To give a proper definition, we need the following lemma.
	
	\begin{lemma}\label{lem:range-of-full-hardy}
		The following statements hold:
		\begin{enumerate}
			\item[\rm i)] If $\KK_{1}$ is a closed subspace of $\KK$, then for each $f\in H^{2}_{\KK}$ we have that \begin{equation}
				(P_{H^{2}_{\KK_{1}}}f)(z) = P_{\KK_{1}} (f(z)) \quad \text{for a.e. } z\in\TT,
			\end{equation}
			where $P_{H^2_{\KK_1}}: H^2_\KK\to H^2_\KK$ is the orthogonal projection of $H^2_{\KK}$ onto $H^2_{\KK_1}$ and $P_{\KK_{1}}:\KK\to\KK$ is the orthogonal projection of $\KK$ onto $\KK_1$.
			\item[\rm ii)] Let $J$  be a range function on $\KK$. Then  $J$  is measurable if and only if the map $\la\mapsto H^{2}_{J(\la)}$ is a measurable range function in $H^{2}_{\KK}
			$. 
		\end{enumerate}
	\end{lemma}
	
	\begin{proof}
		{\rm i)} Let $\cB_1=\{\var_{i}\}_{i \in I}$ be an orthonormal basis of $\KK_1$ (with $\#I=\dim(\KK_1)$). Let $Q:H^2_\KK\to H^2_\KK$ be defined by $(Qf)(z)=P_{\KK_1}(f(z))$ for a.e. $z\in\TT$, and $f\in H^2_\KK$.
		In order to see that $Q$ is a well-defined operator, observe that the $\KK$-valued function $z\mapsto P_\KK(f(z))$ belongs to $H^2_\KK$. Indeed, by completing $\cB_1$ to an orthonormal basis of $\KK$, namely $\cB$, we have that the coordinate functions of $f\in H^2_\KK$ respect to the basis $\cB$ belong to $H^2$ (in particular, $f_i=\langle f(\cdot), \var_i\rangle_\KK \in H^2$ for every $i\in I$). Since $P_\KK(f(z))=\sum_{i\in I} \langle f(z),\var_i\rangle_\KK \var_i$ for a.e. $z\in\TT$, we have that $P_\KK (f(\cdot))\in H^2_\KK$. 
		
		Now, observe that $Q^2=Q$ by computing, for a.e. $z\in\TT$ and $f\in H^2_\KK$,
		\begin{align}
			(Q^2 f)(z)&=(Q(Qf))(z) = P_{\KK_1}((Qf)(z))= P_{\KK_1} (P_{\KK_1}(f(z))) \\
			&= P^2_{\KK_1}(f(z)) = P_{\KK_1}(f(z)) = (Qf)(z).
		\end{align}
		Moreover, we see that $Q^*= Q$, as for every $f,g\in H^2_\KK$:
		\begin{align}
			\langle Qf, g \rangle &= \int_{\TT} \langle (Qf)(z), g(z) \rangle_\KK \,dz 
			= \int_{\TT} \langle P_{\KK_1} (f(z)), g(z) \rangle_\KK \,dz\\
			&= \int_{\TT} \langle f(z), P_{\KK_1} (g(z)) \rangle_\KK \,dz
			=\int_{\TT} \langle f(z), (Qg)(z) \rangle_\KK \,dz = 	\langle f, Qg \rangle.
		\end{align}
		Consequently, we have that $Q$ is an orthogonal projection. Even more, $Q(H^2_\KK)= H^2_{\KK_1}$ since $(Qf)(z)\in \KK_1$ for a.e. $z\in\TT$ and $f\in H^2_\KK$ and $Qf=f$ for every $f\in H^2_{\KK_1}$. 
		Hence, we deduce that $Q=P_{H^2_{\KK_1}}$, which proves statement i).
		
		{\rm ii)} Let $\{\var_{i}\}_{i \in I}$ be an orthonormal basis of $\KK$ (with $\# I=\dim(\KK)$). Then, the system $\{S^j\var_i:j\in\N_0, i \in I\}$ is an orthonormal basis of $H^2_\KK$. 
		Using i) we see that
		\begin{align*}
			\langle P_{H^{2}_{J(\la)}} S^{j} \var_{i},S^{j'} \var_{i'} \rangle
			& = \int_{\TT} \langle (P_{H^{2}_{J(\la)}} S^{j} \var_{i})(z),( S^{j'} \var_{i'} )(z) \rangle_\KK \,dz \\
			& = \int_{\TT} \langle P_{J(\la)} (z^{j} \var_{i}) , z^{j'} \var_{i'} \rangle_\KK \, dz \\
			& = \langle P_{J(\la)} \var_{i} , \var_{i'} \rangle_\KK \int_{\TT}   z^{j-j'} \, dz \\
			&= \delta_{j,j'} \langle P_{J(\la)} \var_{i} , \var_{i'} \rangle_\KK.
		\end{align*}
		From here, statement ii) follows immediately.
	\end{proof}

	\begin{definition}
		Let $J$ be a measurable range function in $\KK$. The {\it full-Hardy} subspace with {\it base} $J$ is the unique closed subspace $\cW\subseteq L^{2}(\TT,H^{2}_{\KK})$ reducing  for $U$ whose range function is given by $\la \mapsto H^{2}_{J(\la)}$ for a.e. $\la\in\TT$.
	\end{definition}
	
	This definition makes sense because, given $J$ a measurable range function in $\KK$, by Lemma \ref{lem:range-of-full-hardy}, we know that $\la \mapsto H^{2}_{J(\la)}$  is a measurable range function in $H^2_\KK$. 
	
	Let us show some properties on these spaces.
	
	\begin{proposition}\label{prop:full-hardy-orthogonal}
		If $\cW\subseteq L^2(\TT,H^2_\KK)$ is a full-Hardy subspace, then so is $\cW^\perp$.
	\end{proposition}
	
	\begin{proof}
		Assume that $\cW$ is full-Hardy with base $J$, for $J$ a measurable range function in $\KK$. By i) in Lemma \ref{lem:props-J}, we have that the measurable range function of $\cW^\perp$  is given by $\la\mapsto (H^2_{J(\la)})^\perp$ for a.e. $\la\in\TT$. Now, since $(H^2_{J(\la)})^\perp=H^2_{J(\la)^\perp}$, if we denote by  $J^{\perp}$ the range function given by $J^{\perp}(\la) = (J(\la))^\perp$  for a.e. $\la\in\TT$ (which is measurable by means of ii) in Lemma \ref{lem:range-of-full-hardy}), we have that $\cW^\perp$ is the full-Hardy subspace with base $J^\perp$.
	\end{proof}
	
	\begin{remark}\label{rem:reducing-S}
		Given any closed subspace $\KK_1\subseteq \KK,$ we have that $H^{2}_{\KK} = H^{2}_{\KK_1}\oplus H^{2}_{\KK_1^{\perp}}$. The fact that both components are invariant under $S$ implies that $H^{2}_{\KK_1}$ is reducing for $S$. Actually, these are the only \textit{reducing} subspaces for $S$ (see, for instance, \cite[Theorem 3.22]{RR}). 
	\end{remark}
	
	As we anticipated, full-Hardy spaces are those that are reducing for $U$ and for $\widehat{S}$. We prove this fact in the next theorem. 
	
	\begin{theorem}\label{prop:W-reduncing-for-U-S}
		Let $\cW\subseteq L^2(\TT,H^2_\KK)$ be a closed subspace. Then, $\cW$ is simultaneously reducing for $U$ and for $\widehat{S}$ if and only if $\cW$ is a full-Hardy subspace.
	\end{theorem}
	
	\begin{proof}
		Let $\cW$ be reducing for $U$ and for $\widehat{S}$ with range function  $J_{\cW}$  in $H^2_{\KK}$. By Lemma \ref{lem:invariance_J_S} and i) in Lemma \ref{lem:props-J}, we have that $J_{\cW}(\la)\subseteq H^2_\KK$ is reducing for $S$ for a.e. $\la\in\TT$ (since $\cW$ and $\cW^\perp$ are invariant under $\widehat{S}$, then $J_\cW(\la)$ and $J_{\cW^\perp}(\la)=(J_{\cW}(\la))^\perp$ are invariant under $S$ a.e. $\la\in\TT$). 
		As discussed in Remark \ref{rem:reducing-S}, this implies that $J_{\cW}(\la)= H^2_{\KK_\la}$, where $\KK_\la$ is a closed subspace of $\KK$ for a.e. $\la\in\TT$. By ii) in Lemma \ref{lem:range-of-full-hardy},  since $J_\cW:\la \mapsto H^2_{\KK_\la}$ is measurable, then so is $J:\la \mapsto \KK_\la$. We conclude that $\cW$ is full-Hardy with base  $J$.
		
		Conversely, if $\cW$ is full-Hardy with base $J$, for $J$ a measurable range function in $\KK$, by Proposition \ref{prop:full-hardy-orthogonal} and Lemma \ref{lem:invariance_J_S}, we have that $\cW$ and $\cW^\perp$ are clearly invariant under $\widehat{S}$. Then, $\cW$ is reducing for~$\widehat{S}$.
	\end{proof}
	
	\begin{remark}\label{rem:full-hardy}
		Let $J$ be a measurable range function in $\KK$, let $\cW\subseteq L^2(\TT,H^2_\KK)$ be a full-Hardy subspace with base $J$ and let $\cR$ be the wandering subspace for $\widehat{S}$ associated to $\cW$, that is $\cR=\cW\ominus\widehat{S}\cW.$ By Remark \ref{rem:R_0-S-sombrero}, we know that $\cR$ is reducing for $U$ and that its range function is given by 
		$$J_\cR(\la)= J_{\cW}(\la)\ominus S(J_\cW(\la))=H^2_{J(\la)}\ominus S(H^2_{J(\la)})=J(\la)$$ for a.e. $\la\in\TT$, where $J(\la)\subseteq \KK \subseteq H^2_\KK$ is understood as a subspace of $H^2_\KK$ of constant functions. This implies that the wandering subspace for $\widehat{S}$ associated to a full-Hardy subspace can be seen as a subspace of $L^2(\TT,\KK)\subseteq L^2(\TT,H^2_\KK)$ reducing for $U$.
	\end{remark}
	
	Combining this with ii) of Remark \ref{rem:R_0-S-sombrero}, we see that given two full-Hardy subspaces $\cW_1,\cW_2$ with respective bases $J_1,J_2$  such that $\cW_1\simeq_\Phi\cW_2$, where $\Phi:L^2(\TT,H^2_\KK) \to L^2(\TT,H^2_\KK)$ is an isometry  that  commutes with $\widehat{S}$ and $U$, then $J_1(\la)\simeq J_2(\la)$ isometrically for a.e. $\la\in\TT$.

	\subsection{Subspaces reducing for $U$ and invariant under $\widehat S$}\label{characterization}
	In Theorem \ref{prop:W-reduncing-for-U-S} we proved that full-Hardy subspaces are the only  reducing subspaces for $U$ that are also reducing for $\widehat{S}$ (in particular, invariant under $\widehat{S}$). Theorem \ref{characterization_hatS_inv}
	shows that any other subspace reducing for $U$ and invariant under $\widehat{S}$ needs to be isometrically isomorphic to a full-Hardy subspace.

	\subsubsection{Proof of Theorem \ref{characterization_hatS_inv}}
	
	\begin{proof} 
		i) $\Rightarrow$ ii) 
		We wish to find a full-Hardy subspace $\cW\subseteq L^{2}(\TT,H^{2}_{\KK})$ and a partial isometry $\Phi:L^{2}(\TT,H^{2}_{\KK}) \rightarrow L^{2}(\TT,H^{2}_{\KK})$ with initial space $\cW$ that commutes with $U$ and $\widehat{S}$ such that $\cW\simeq_\Phi\cM$. 
		Notice that since $\Phi$ must commute with $U$, by Proposition \ref{commutant_U_S}, $\Phi\in \widehat{\cF}$, that is, $\Phi = \widehat{F}$ for $F:\TT\to \cB(H^2_\KK)$ an operator-valued function in the class $\cF$. Let $\cR$ be the wandering subspace for $\widehat{S}$ associated to $\cM$, i.e. $\cR=\cM\ominus\widehat{S}\cM$. From the discussions in item ii) of Remark \ref{rem:R_0-S-sombrero} and Remark \ref{rem:full-hardy}, we deduce that the basis range function $J$ of $\cW$ should satisfy that $J(\la)\simeq_{F(\la)} J_\cR(\la)$ for a.e. $\la\in\TT$.
		
		Thus, the idea of the proof will be to construct first a measurable range function $J$ in $\KK$ such that $\dim(J(\la)) = \dim(J_{\cR}(\la))$ for a.e. $\la\in\TT$, and secondly, a measurable operator-valued function $F:\TT\to\cB(H^2_\KK)$ such that for a.e. $\la\in\TT$, $F(\la)$ is a partial isometry with initial space $H^{2}_{J(\la)}$ that commutes with $S$ and satisfies that 
		\begin{equation}\label{eq:initial-space-F}
			F(\la)(H^{2}_{J(\la)}) = J_{\cM}(\la).
		\end{equation}    
		Because the dimension of $J_\cR(\la)$ may vary with $\la\in\TT$, to pursue these constructions we will consider the measurable sets $\{A_{n}\}_{n\in \N_0\cup\{\infty\}}$ and the functions $\{\phi_{i}\}_{i\in\N}\subset L^{\infty}(\TT,H^{2}_{\KK})$ provided by Lemma~\ref{measurable_sets} applied to the subspace $\cR$.

		Fix $\{\var_i\}_{i\in I}$ an orthonormal basis of $\KK$ where $I=\{1, \dots, k\}$ in case that $\dim(\KK)=k\in\N$ or $I=\N$ if $\dim(\KK)=\infty$.   For every $n\in I$, let us define the subspaces $\KK_{n} = \spn \{\varepsilon_{1},\dots,\varepsilon_{n}\}$ and $\KK_0=\{0\}$. 
		We will construct the range function $J$ in $\KK$ as follows: 
		$$J(\la) =\begin{cases}
			\KK_{n}\quad &\text {if }\la \in A_n,\, n\in I\\
			\KK &\text {if }\la \in A_{\infty}.
		\end{cases}$$
		It is clear that $J$ is measurable since it is constant on each measurable set $A_n$ in the partition of $\TT$.
		In addition, we have that $\dim (J_{\cR}(\la) )= \dim (J(\la))$ for a.e. $\la \in \TT$ by definition of $J$, because $A_n=\{\la\in\TT\,:\,\dim (J_{\cR}(\lambda))=n \}$. We take $\cW\subseteq L^2(\TT,H^2_\KK)$ as the full-Hardy space with base $J$.
		
		To construct $F$ as in \eqref{eq:initial-space-F}, we will proceed in the following way: 
		for a.e. $\la\in A_0$, define $F(\la)\equiv0$. Now, let us fix $n\in\N$. Observe that from \eqref{eq:oplus-S^jR_0} we can deduce that, for a.e. $\la\in A_n$, the system $\{S^{j} \phi_{i}(\la):\,i=1,\dots,n,\,j\in\N_0\}$ is an orthonormal basis of $J_{\cM}(\la)$. On the other hand, for a.e. $\la\in A_n$, we have that $\{S^j\var_i:\,i=1,\dots,n,\,j\in\N_0\}$ is an orthonormal basis of $H^2_{J(\la)}$. Thus, for a.e. $\la\in A_n$, we define $F(\la):H^2_{J(\la)}\to J_\cM(\la)$ as
		\begin{equation}\label{eq:partial-isometry}
			F(\la)( S^{j} \var_{i} ) = S^{j}\phi_{i}(\la),\quad i=1,\dots,n, \, j\geq 0,
		\end{equation}
		extended by linearity to the whole $H^2_{J(\la)}$. 
		For $n=\infty$, we proceed similarly because, for a.e. $\la\in A_\infty$,  the system $\{S^{j} \phi_{i}(\la):\,i\in\N,\,j\in\N_0\}$ is an orthonormal basis of $J_{\cM}(\la)$,  and  
		$\{S^j\var_i:\,i\in\N,\,j\in\N_0\}$ is an orthonormal basis of $H^2_{J(\la)}$. Thus, 
		we define $F(\la)$ as in \eqref{eq:partial-isometry} but for every $i\in\N$.  
		
		Now, for a.e. $\la\in\TT$, we set $F(\la)f = 0$ for every $f$ in $(H^{2}_{J(\la)})^\perp\subseteq H^2_\KK$, getting that the operator $F(\la):H^{2}_{\KK} \rightarrow H^{2}_{\KK}$ is a partial isometry with initial space $H^{2}_{J(\la)}$ and equation \eqref{eq:initial-space-F} holds. Finally, it is clear that $F(\la)$ commutes with $S$ for a.e. $\la\in \TT$.
		
		Let us see that the operator-valued function $F:\TT\to \cB(H^2_\KK)$ constructed above belongs to the class $\cF$. Indeed, it satisfies that $\|F\|_\infty<\infty$ as $\|F(\la)\|_{op}\leq 1$ for a.e. $\la\in\TT$.
		To show that it is measurable we need to see that the complex-valued function
		\begin{equation}\label{eq:measurability-F}
			\la\mapsto \left\langle F(\la) f,g \right\rangle_{H^{2}_{\KK}} = \langle F(\la) P_{H^2_{J(\la)}}f,g \rangle_{H^{2}_{\KK}} 
		\end{equation} is measurable for every $f,g\in H^{2}_{\KK}$. 
		Again, let us prove this over each set of the measurable partition $\{A_{n}\}_{n\in \N_0\cup\{\infty\}}$ of $\TT$. On the set $A_0$ it is clear that $F(\la)\equiv 0$ is measurable. Now, fix $n\in\N$. For a.e. $\la\in A_n$ and for $f\in H^2_\KK$, we have that
		\begin{equation}
			P_{H^{2}_{J(\la)}}f = \sum_{j=0}^{\infty}\sum_{i=1}^{n} \left\langle f, S^{j} \var_{i} \right\rangle_{H^{2}_{\KK}} S^{j} \var_{i},
		\end{equation}
		and therefore
		\begin{align*}
			\left\langle F(\la) P_{H^{2}_{J(\la)}}f,g \right\rangle 
			&=  \sum_{j=0}^{\infty}\sum_{i=1}^{n} \left\langle f, S^{j}\var_{i} \right\rangle_{H^{2}_{\KK}}  \left\langle F(\la) (S^{j}\var_{i}) ,g \right\rangle_{H^{2}_{\KK}} \\
			&=  \sum_{j=0}^{\infty}\sum_{i=1}^{n}  \left\langle f, S^{j} \var_{i} \right\rangle_{H^{2}_{\KK}}  \left\langle S^{j}\phi_{i}(\la) ,g \right\rangle_{H^{2}_{\KK}}. 
		\end{align*} 
		Then, \eqref{eq:measurability-F} is measurable for every $f,g\in H^2_\KK$ since so is $$\la\mapsto \left\langle S^{j}\phi_{i}(\la) ,g \right\rangle_{H^{2}_{\KK}} 
		= \left\langle (\widehat{S}^{j}\phi_{i})(\la),g \right\rangle_{H^{2}_{\KK}} $$ for every $j\in\N_0$ and $i=1,\dots,n$. 
		
		If $n=\infty$, we proceed in the same way but taking into account that 
		\begin{equation}
			P_{H^{2}_{J(\la)}}f = \sum_{j=0}^{\infty}\sum_{i=1}^{\infty} \left\langle f, S^{j} \var_{i} \right\rangle_{H^{2}_{\KK}} S^{j} \var_{i}.
		\end{equation}
		
		Let $\Phi: L^{2}(\TT,H^{2}_{\KK}) \rightarrow L^{2}(\TT,H^{2}_{\KK})$ be the function defined as $\Phi = \widehat{F}$ (an thus commuting with $U$). Since $F(\la)$ commutes with $S$ then it is easily seen that $\Phi$ commutes with $\widehat{S}$. Finally, by i) and iii) in Lemma \ref{lem:props-range-op}, since $F(\la)$ is an isometry with initial space $H^{2}_{J(\la)}$ and \eqref{eq:initial-space-F} holds for a.e. $\la\in\TT$, we conclude that $\Phi$ is a partial isometry with initial space $\cW$ and $\Phi(\cW) = \cM$. 
		
		ii) $\Rightarrow$ i) Assuming that $\cW$ is a full-Hardy space with base range function $J$, we have that $\cW$ is reducing for $U$ and $\widehat{S}$ (see Theorem \ref{prop:W-reduncing-for-U-S}). Then, as $\Phi$ commutes with $\widehat{S}$ and $\cM = \Phi(\cW)$ we get that
		$$\widehat{S}(\cM) = \widehat{S}\Phi(\cW) =  \Phi \widehat{S}(\cW) \subseteq \Phi(\cW) = \cM,$$
		that is, $\cM$ is invariant under $\widehat{S}$. Analogously, it can be seen that $\cM$ is reducing for $U$ given that $\Phi$ also commutes with $U$ and $U^{*}$.
	\end{proof}

	The fact that the initial subspace of the partial isometry in Theorem \ref{characterization_hatS_inv}
	is full-Hardy is a consequence of the commutation with $U$ and $\widehat{S}$ as the next proposition shows.
	
	\begin{proposition}
		If $\Phi: L^2(\TT,\KK) \rightarrow  L^2(\TT,\KK)$  is a partial isometry that commutes with $U$ and $\widehat{S}$, then the initial space of $\Phi$ is full-Hardy.
	\end{proposition} 
	
	\begin{proof}
		We need the following property:
		
		a) \label{full}
		If $T,V :\HH\rightarrow \HH$, $T$ is an isometry  and $V$ a partial isometry that commutes with $T$, then the initial space of $V$ reduces $T$.
		
		To prove a) set $\cM$ the initial space of $V$ and observe that if $f \in \cM^{\perp}$ then $0=TVf = VTf$  implies
		that $Tf \in {\text Ker}(V) = \cM^\perp$. Thus $\cM^\perp$ is invariant under $T$.
		On the other hand, if $f\in \cM$ then,
		$$ ||VTf|| = ||TVf|| = ||Vf|| =||f|| =||Tf||.$$
		So, $V$ preserves the norm of $Tf$ which implies that $Tf$ is in the initial space of $V$.
		We conclude that  $\cM$ reduces $T$ which ends the proof of a).
		
		Now, let $\cW \subseteq L^2(\TT,H^2_{\KK})$ be the initial space of $\Phi.$
		We apply a) twice. First to $\Phi$ and $U$, from where we conclude that  $\cW$ reduces $U$. Second to $\Phi$ an $\widehat{S}$ which gives that $\cW$ reduces $\widehat{S}$.
		Now, by Theorem \ref{prop:W-reduncing-for-U-S} in this paper, $\cW$ is full-Hardy.
	\end{proof}

	Taking $\KK=\CC$, we can deduce a characterization of the closed subspaces of $L^2(\TT,H^2)$ which are reducing for $U$ and invariant under $\widehat{S}$, which bear a resemblance to Beurling's Theorem (Theorem \ref{thm:beurling}).
	
	\begin{corollary}\label{M_reducingU_invariantS}
		Let $\mathcal{M}\subseteq L^{2}(\TT,H^{2})$ be a closed subspace. The following statements are equivalent: 
		\begin{enumerate}
			\item[\rm i)] $\mathcal{M}$ is reducing for $U$ and invariant under $\widehat{S}$.		
			
			\item[\rm ii)]  There exists $\phi \in L^{2}(\TT,H^{2})$ such that $\phi(\lambda)$ is an inner function for a.e. $\la \in \sigma(\cM)$ and $\mathcal{M} = \phi \, L^{2}(\TT,H^{2})$, where $\sigma(\cM)$ is defined in \eqref{eq:specturm}.
		\end{enumerate}
		
	\end{corollary}
	
	\begin{proof}
		i) $\Rightarrow$ ii) By taking $\KK=\CC$ in Theorem \ref{characterization_hatS_inv}, there exist a full-Hardy space $\cW\subseteq L^{2}(\TT,H^{2})$ and $\Phi:L^{2}(\TT,H^{2}) \rightarrow L^{2}(\TT,H^{2})$ a partial isometry with initial space $\cW$ that commutes with $U$ and $\widehat{S}$ such that $\cM=\Phi(\cW)$. 
		
		First, observe that since $\Phi=\widehat{F}$, where $F:\TT\to\cB(H^2)\in\cF$ is an operator-valued function such that $F(\la)$ is a partial isometry with initial space $J_\cW(\la)$ and  \eqref{eq:initial-space-F} holds, then $J_\cM(\la)$ and $J_\cW(\la)$ are isomporphic a.e. $\la\in\TT$ (see items i) and iii) of Lemma \ref{lem:props-range-op}). This implies that $E:=\sigma(\cM) = \sigma(\cW)$.
		Let $J$ be the range function in $\CC$ for which $J_\cW(\la)=H^2_{J(\la)}$ for a.e. $\la\in\TT$.
		We observe that  the measurable range function $J$ in $\CC$ has the form $J(\la) = \cX_{E}(\la) \, \CC$.
		Then, $H^{2}_{J(\la)} = \cX_{E}(\la)H^{2}$ for a.e. $\la\in\TT$, which implies that 
		\begin{equation}
			\cW =\{ f\in L^{2}(\TT,H^{2}) : f(\la) \in \cX_{E}(\la) H^{2} \} = \cX_{E} L^{2}(\TT,H^{2}).
		\end{equation}

		Moreover, since $F(\la)$ commutes with $S$, for a.e. $\la\in\TT$, there exists a function $h_{\la} \in H^{\infty}$ such that $F(\la)=M_{h_{\la}}$, where $M_{h_\la}(f)=h_\la f$ for $f\in H^2$ (see \cite[Theorem 3.4]{RR}).  Further, as $F(\la)$ is a partial isometry with initial space $\cX_E(\la)H^2$ for a.e. $\la\in\TT$, then $h_\la\equiv 0$ for a.e.  $\la\in\TT\setminus E$ and, for a.e. $\la\in E$, it holds that $|h_{\la}(z)| = 1$ (that is, $h_\la$ is an inner function).
		Using the measurability of $\la\mapsto F(\la)$ we get that $\la \mapsto h_{\la} = F(\la)1$ is an $H^2$-valued measurable function. 
		Hence, if we define $\phi(\la)= h_{\la}$ for a.e. $\la\in\TT$, we obtain that $\phi\in L^2(\TT,H^2)$. In fact, it can be seen that $\Phi =  M_\phi$, where $M_\phi (f) = \phi f$ for every $f\in L^2(\TT,H^2)$. Indeed, for $f\in L^2(\TT,H^2)$ and for a.e. $\la\in\TT$,
		\begin{equation}
			(\Phi f)(\la) = F(\la) f(\la) = h_\la f(\la) = \phi(\la) f(\la) = (\phi f)(\la) = (M_\phi(f)) (\la).
		\end{equation} 
		Consequently, $\cM = \Phi(\cW) = \phi \cX_E L^{2}(\TT, H^{2}) = \phi L^{2}(\TT, H^{2}).$
		
		ii)$\Rightarrow$ i) Assuming that $\cM = \phi L^2(\TT,H^2)$ with $\phi \in L^{2}(\TT,H^{2})$ such that $\phi(\lambda)$ is an inner function for a.e. $\la \in \sigma(\cM)$, then 
		\begin{equation}
			\widehat{S}(\cM) = \widehat{S}\phi L^2(\TT,H^2) =\phi  \widehat{S}L^2(\TT,H^2) \subseteq \phi L^2(\TT,H^2) = \cM.
		\end{equation}
		Analogously, it can be seen that $U\cM\subseteq \cM$ and $U^*\cM\subseteq \cM$. 
	\end{proof}
	
	We finish this section with a result that describes the range functions associated to subspaces in $L^2(\TT, H^2)$ of the form $\phi L^2(\TT, H^2) $ for some $\phi\in L^2(\TT, H^2)$. In particular, this shows how are the range functions associated to subspaces that are reducing for $U$ and invariant for $\widehat S$.

	\begin{proposition}\label{porp:range-function-phiL2}
		Let  $\phi \in L^{2}(\TT,H^{2})$. If $\mathcal{M} = \phi \, L^{2}(\TT,H^{2})$, 
		then the range function associated to $\cM$ is given by 
		$$J_\cM(\la)=\phi(\la)H^2 \quad \textrm{ a.e. }\la\in\TT.$$
	\end{proposition}
	
	\begin{proof}
		Let $J$ be the measurable range function given by $J(\la)=\phi(\la)H^2$ a.e. $\la\in\TT$ and let $\cN$ be the subspace of $L^{2}(\TT,H^{2})$ whose range function is $J$. We will prove that $\cM=\cN$.
		
		For $g\in L^{2}(\TT,H^{2})$, we have that $f:=\phi g\in\cM$. Then, since $f(\la)=\phi(\la) g(\la)\in\phi(\la)H^2$ a.e. $\la\in\TT$, it holds that $J_\cM(\la)\subseteq  \phi(\la)H^2$ a.e. $\la\in\TT$ and then $\cM\subseteq  \cN$. 
		
		To prove that the other inclusion also holds, assume that there is $f\in\cN$ such that $f\perp\cM$. 
		Notice that, since $f(\la)\in \phi(\la) H^2$ for a.e. $\la\in\TT$ we can write $f(\la)=\phi(\la)h_\la$ for some $h_\la\in H^2$ for a.e $\la\in\TT$. 
		Therefore, since $\phi\in\cM$ and, by Lemma \ref{lem:props-J}, $f(\la)\in J_\cM(\la)^\perp$ for 
		a.e. $\la\in\TT$, we obtain that 
		$$ 0=\langle \phi(\la), f(\la)\rangle_{H^2} = h_\la\|\phi(\la)\|_{H^2}^2$$
		for a.e. $\la\in\TT$. Then, $h_\la=0$ for a.e $\la\in\supp(\phi)$. Hence $f(\la)=0$ for a.e. $\la\in\TT$ and thus $f=0$ proving that $\cM=\cN$.
	\end{proof}
	
	%%%%%%%%%%%%%%%%%%%%%%%%%%%%%%%%%%%%%%%%%%%%%%%%%%%%%%%%%
	
	\subsection{Uniqueness}\label{uniqueness}
	
	We will prove here the uniqueness of the characterizations in Theorem \ref{characterization_hatS_inv} and Corollary \ref{M_reducingU_invariantS}.
	
	\subsubsection{Proof of Theorem  \ref{thm:uniqueness}}
	\begin{proof} 
		Let $\Psi := (\Phi_1|_{\cW_{1}})^{-1} \Phi_2$.	
		This operator is well-defined and bounded acting on $L^{2}(\TT,H^{2}_{\KK})$ and satisfies that $\Psi(\cW_2)=\cW_1$ and $\Phi_2=\Phi_1\Psi$. In fact, it is a partial isometry with initial space $\cW_2$. 
		Indeed, for $f\in\cW_2$, we have that
		$\|f\| = \|\Phi_2(f)\|=\|\Phi_1(\Psi(f))\|=\|\Psi(f)\|,$
		and for $f\in\cW_2^\perp$,  $\Psi(f)=0$.
		
		Let us see now that  $\Psi$ commutes with $U$ and $\widehat{S}$.
		Since $\Phi_2=\Phi_1\Psi$ and using the commutativity of $\Phi_1$ and $\Phi_2$ with $U$, for every $f\in L^{2}(\TT,H^{2}_{\KK})$ we have
		\begin{equation}
			\Phi_2U f = \Phi_1\Psi\, U f \quad \text{and} \quad
			\Phi_2U f = U\Phi_2 f = U\Phi_1\Psi f = \Phi_1U \Psi f .
		\end{equation}
		So we conclude that $\Phi_1\Psi U f = \Phi_1U \Psi f$. Now, if $f \in \cW_2$ then $U \Psi f$ and $\Psi U f $ belong to $\cW_1$ and, since $\Phi_1$ is one-to-one in $\cW_1$,  we have that $\Psi U f = U \Psi f.$
		On the other hand, if $f\in \cW_2^{\perp}$ then $U \Psi f = 0 = \Psi U f$ because  $\cW_2$ is the initial space of $\Psi$  and is reducing for $U$.
		The proof for $\widehat{S}$ is the same replacing $U$ by $\widehat{S}$ above.
	\end{proof}
	
	\begin{theorem}
		Let $\phi_1,\phi_2\in L^2(\TT,H^2)$ be two functions such that $\phi_i(\la)$ is an inner function for a.e. $\la\in E_i:=\text{Supp}\, \phi_i$ for $i=1,2$ and $\phi_1L^2(\TT,H^2)=\phi_2 L^2(\TT,H^2)$. Then,  $E_1=E_2:=E$ and there exists a function $\psi\in L^2(\TT,H^2)$ such that $\psi(\la)$ is an inner function for a.e. $\la\in E$ and $\phi_2 = \psi \phi_1$.
	\end{theorem}
	
	\begin{proof}
		As $\phi_1L^2(\TT,H^2)=\phi_2 L^2(\TT,H^2)$, it is clear that $E_1=E_2=E$.
		Moreover, by Proposition \ref{porp:range-function-phiL2} for a.e. $\la\in E$, we have that
		$$\phi_{1}(\lambda)H^{2} = \phi_{2}(\lambda) H^{2}.$$
		As for $\la\in E$,  $\phi_{1}(\lambda)$, $\phi_{2}(\lambda)$ are  inner functions, by Theorem \ref{invariant_subspace_Hardy}, we obtain that for a.e. $\la\in E$, $\phi_{1}(\lambda)/\phi_2(\la)=c_\la$, where $c_{\lambda}\in\CC$ with $|c_{\lambda}|=1$.
		
		Since $\phi_{1}$, $\phi_{2}$ are measurable functions, the function  $\psi:\TT\rightarrow \CC$, given by $\psi(\lambda)=c_{\lambda}$ for $\la\in E$ and $\psi(\la)=0$ for $\la\in \TT\setminus E$, is measurable. Also $\psi(\la)$ is inner for a.e. $\la\in E$.
	\end{proof}

	\subsection{Multiplication-invariant subspaces}\label{MIS}
	
	In this section, we discuss how the results we present in this paper can be stated with more generality, by replacing the circle $\TT$ with a $\sigma$-finite measure space $(X,\mu)$ for which $L^2(X):=L^2(X,\mu)$ is separable. Given a Hilbert space $\KK$, we define $L^2(X,\KK)$ in the same manner as in Subsection \ref{sec:vector-valued-functions}, where the vector-valued functions are defined on $X$.
	
	\begin{definition}\label{def:mult-inv}
		Let $\cM$ be a closed subspace of $L^2(X,\KK)$. We say that $\cM$ is multiplication invariant if for every $f\in\cM$ and $g\in L^\infty(X)$, it holds that $fg \in \cM$.
	\end{definition}
	
	To see that this property holds, it suffices to check the multiplication invariance over a subset of $L^\infty(X)$ called determining set. A subset $\mathcal D$ of $L^\infty(X)$ is said to be a determining set for $L^1(X)$ if for every $f\in L^1(X)$, it holds that
	$$\int_X fg \,d\mu = 0 \quad \forall g\in \mathcal D \quad \Rightarrow \quad f=0.$$
	
	For instance, when $X=\TT$ and $\mu$ is the normalized Lebesgue measure, the system $\{\xi^j\}_{j\in\Z}$ 
	is a determining set for $L^1(\TT)$ where $\xi(\la)=\la$, $\la\in\TT$. Observe, then, that a subspace $\cM\subseteq L^2(\TT,\KK)$ is reducing for the bilateral shift $U$ if it is multiplication invariant, given that it is invariant with respect to the determining set $\{\xi^j\}_{j\in\Z}$.
	
	\begin{definition}\label{def:mult-inv-op}
		We say that an operator $\Phi:L^2(X,\KK)\to L^2(X,\KK)$ is multiplication invariant if it is bounded and for every $\phi\in L^\infty(X)$, $\Phi$ commutes with the multiplication operator $M_\phi:L^2(X,\KK)\to L^2(X,\KK)$, defined by $M_\phi(f)=\phi f$ for $f\in L^2(X,\KK)$. 
	\end{definition}
	
	As before, if $\mathcal D$ is  a determining subset for $L^1(X)$, to see that $\Phi$ is multiplication invariant, it suffices to show that $\Phi$ commutes with $M_\phi$ for every $\phi\in \mathcal D$.
	
	In the case of $X=\TT$, it is clear that a bounded operator $\Phi:L^2(X,\KK)\to L^2(X,\KK)$ is multiplication invariant if and only if $\Phi$ commutes with the bilateral shift $U$.
	
	The proof of Theorem \ref{characterization_hatS_inv} is constructed over two main concepts. On the one hand, the existence of a range function for the subspaces of $L^2(\TT,\KK)$ which are reducing for $U$. This existence is assured by Theorem \ref{reducing_U} that, as we mention before, still holds in the general case of $L^2(X,\KK)$, see \cite[Theorem 2.4]{BR}, allowing the ensuing results of Subsection \ref{sec:reducing-subspaces} to be stated in the general setting.
	On the other hand, the existence of an operator-valued function $F:\TT\to \cB(\KK)$ associated to an operator in $\cB(L^2(\TT,\KK))$ that commutes with $U$ as stated in Proposition \ref{commutant_U_S}. This result can be extended to the general setting by \cite[Theorem 3.7]{BI} which shows that a multiplication invariant operator $\Phi:L^2(X,\KK)\to L^2(X,\KK)$ admits a range operator \cite[Definition 3.6]{BI}.
	
	Moreover, the concept of full-Hardy space that we introduced in Section \ref{full-Hardy}  can be stated in $L^2(X, H^2_\KK)$. Therefore, we obtain the following version of Theorem \ref{characterization_hatS_inv} in $L^2(X, H^2_\KK)$.  
	\begin{theorem}
		Let $\mathcal{M}\subseteq L^{2}(X,H^{2}_{\KK})$ be a closed subspace. The following statements are equivalent:
		\begin{enumerate}
			\item[\rm i)] $\mathcal{M}$ is mutiplication invariant  and invariant under $\widehat{S}$.
			\item[\rm ii)] There exists  a full-Hardy subspace $\cW\subseteq L^{2}(X,H^{2}_{\KK})$ and a partial isometry \\ $\Phi:L^{2}(X,H^{2}_{\KK}) \rightarrow L^{2}(X,H^{2}_{\KK})$ with initial space $\cW,$  that is multiplication invariant  and commutes with $\widehat{S}$
			such that $\cW \simeq_{\Phi} \cM.$
		\end{enumerate}
	\end{theorem} 
	
	\section{Acknowledgements}
	This research was supported by grants: UBACyT 20020170100430BA, PICT 2018-3399 (AN- PCyT), PICT 2019-03968 (ANPCyT) and CONICET PIP 11220110101018. D.C. was supported by the European Union’s programme Horizon Europe, HORIZON-MSCA-2021-PF-01, Grant agreement No. 101064206.

\end{document}